\documentclass[a4paper,reqno,12pt]{amsart}

\usepackage[english]{babel}
\usepackage[utf8]{inputenc}
\usepackage[T1]{fontenc}
\usepackage{amssymb,amsmath,amsthm}
\usepackage[colorlinks=true, urlcolor=blue, linkcolor=blue, citecolor=blue]{hyperref}
\usepackage{graphicx,color}
\usepackage[all]{xy}
\numberwithin{equation}{section}
\usepackage[style=alphabetic,
maxnames=50,maxalphanames=4]{biblatex}
\usepackage{caption}
\usepackage[all]{xy}
\usepackage{wasysym}
\usepackage{geometry}
\usepackage[dvipsnames]{xcolor}
\usepackage{thm-restate}
\usepackage{fancyvrb}
\usepackage{comment}
\usepackage{cleveref}
\renewbibmacro{in:}{}
\theoremstyle{plain}
\newtheorem{Theorem}{Theorem}[section]
\newtheorem{Corollary}[Theorem]{Corollary}
\newtheorem{Lemma}[Theorem]{Lemma}
\newtheorem{Proposition}[Theorem]{Proposition}

\theoremstyle{definition}
\newtheorem{Definition}[Theorem]{Definition}
\newtheorem{Example}[Theorem]{Example}
\newtheorem{Remark}[Theorem]{Remark}

\def\CC{{\mathbb C}}

\def\ker{{\rm ker}\,}

\def\sym#1#2{\mbox{\rm Sym}_{#1}(#2)}

\def\sym{{\mathcal S}}

\def\cl#1{{\mathcal #1}}

\def\VaVa{{\cl V}\kern-5pt {\cl V}}

\def\PP{{\mathbb P}}

\def\C{{{\mathbb C}}}

\newcommand{\blue}{\color{blue}}

\begin{document}
\author{Giorgio Ottaviani}

\address{Dipartimento di Matematica e Informatica ``Ulisse Dini'',  University of Florence, viale Morgagni 67/A, I-50134, Florence, Italy}
\email{giorgio.ottaviani@unifi.it}
\author{Ettore Teixeira Turatti}
\address{Department of Mathematics and Statistics, UiT The Arctic University of Norway, Forskningsparken 1 B 485, Troms\o, Norway}
\email{ettore.t.turatti@uit.no}


\title{Generalized identifiability of sums of squares}

\begin{abstract}
Let $f$ be a homogeneous polynomial of even degree $d$. We study the decompositions $f=\sum_{i=1}^r f_i^2$ where $\deg f_i=d/2$. The minimal number of summands $r$ is called the $2$-rank of $f$,
so that the polynomials having $2$-rank equal to $1$ are exactly the squares.
Such decompositions are never unique and they are divided into $\mathrm{O}(r)$-orbits, the problem becomes counting how many different $\mathrm{O}(r)$-orbits of decomposition exist. We say that $f$ is $\mathrm{O}(r)$-identifiable if there is a unique $\mathrm{O}(r)$-orbit. We give sufficient conditions for generic and specific $\mathrm{O}(r)$-identifiability. Moreover, we show the generic $\mathrm{O}(r)$-identifiability of ternary forms. 
\end{abstract}
\maketitle

\section{Introduction}

Let $\sym^d\CC^{n+1}=\CC[x_0,\dots,x_n]_d$ be the vector space of homogeneous polynomials of degree $d$. Suppose $d$ is even, then for every $f\in \sym^d\CC^{n+1}$, there exists a minimal number $r$ and polynomials $f_1,\dots,f_r\in \sym^{d/2}\CC^{n+1}$ such that $\label{eq: decomp}
f=\sum_{i=1}^{r}f_i^2.$ This decomposition is named a decomposition of $f$ as sum of squares. The sum of squares decomposition has a huge interest in applications since its real version is a certificate of nonnegativity for polynomials of even degree {\cite{Mohab, Blek}}.

\begin{Definition}
Denote $\mathrm{Sq}_{d,n}=\{g^2\mid g\in \sym^{d/2}\CC^{n+1}\}$ the variety of squares,
 which consist of all polynomials having $2$-rank equal to one. In classic terms, $\mathrm{Sq}_{d,n}$ can be understood as the affine cone over $\mathrm{Sq}_{d,n}^\PP=sq(\PP\sym^{d/2}\CC^{n+1})$, where $sq$ is the \emph{square} embedding $$sq:\PP\sym^{d/2}\CC^{n+1}\to \PP\sym^d\CC^{n+1},\ [f]\mapsto [f^2].$$

The $2$-rank of $f$, denoted $\mathrm{rk}_2(f)$, is the smallest $r$ such that $$f\in \sigma_r^\circ(\mathrm{Sq}_{d,n})\colon =\left\{\sum_{i=1}^{r} g_i^2\mid g_i\in \sym^{d/2}\CC^{n+1} \right\}.$$ In the above definition some $g_i$ are allowed to be zero, so that all sums with $r'\le r$ summands belong to $\sigma_r^\circ(\mathrm{Sq}_{d,n})$. Moreover, notice that since the polynomials are defined over $\CC$, we do not need to consider scalars $\lambda_i$ multiplying $g_i^2$, since in $\sum_{i=1}^r \lambda_ig_i^2$ the $\lambda_i$ can be absorbed into $g_i^2$.
{The affine $r$-secant variety
 is by definition the following closure
 $$\sigma_r(\mathrm{Sq}_{d,n})\colon=\overline{\sigma_r^\circ(\mathrm{Sq}_{d,n}).}$$
 }

\end{Definition}

The generic $2$-rank is the smallest number $r_g$ such that $\sigma_{r_g}(\mathrm{Sq}_{d,n})=\sym^d\CC^{n+1}$. An upper bound for the {generic} $2$-rank $r_g$ is $2^n$, and for fixed $n$ this bound is optimal for all sufficiently large $d$ \cite[Theorem 4]{FOS}.

\begin{Remark}
    Usually the $r$-secant variety is defined in the projective setting, however, due to our interest in the decompositions, it is more natural to consider the affine cone over the $r$-secant variety. Therefore, throughout the text $\sigma_r(\mathrm{Sq}_{d,n})$ will denote the affine cone over $\sigma_r(\mathrm{Sq}_{d,n}^\PP)$, the projective $r$-secant variety of $\mathrm{Sq}_{d,n}^\PP$.
\end{Remark}

In the case of Waring decompositions $\sum f_i^d$, where $\deg f_i=1$, it is known that apart from special cases the decomposition is generically identifiable (namely unique up to permutations and scalar multiplications of $f_i$ by $d$-roots of unity) for polynomials of subgeneric rank \cite{COV}. The identifiability for decompositions $\sum f_i^k$, where $\deg f_i=d/k$ and $d$ is divisible by $k$, has been studied recently in \cite{BCOM, CP} for $3\leq k<d$. 

The case $k=2$ we study in this paper is particular.  The space of all minimal sums of squares decompositions of $f\in\sym^d\CC^{n+1}$ was denoted as $\mathrm{SOS}_r(f)=\{(f_1,\dots,f_r)\in\prod_{i=1}^r\sym^{d/2}\C^{n+1}|f=\sum_{i=1}^rf_i^2\}$ in \cite{FOST}, notice that each decomposition with $r$ squares as summands
create a whole $\mathrm O(r)$-orbit of decompositions, by the trivial identity
\begin{equation}\label{eq:O(r)orbit}\sum_{i=1}^r\left(\sum_{j=1}^rg_{ij}f_j\right)^2=\sum_{i=1}^rf_i^2\end{equation}
for every $\left(g_{ij}\right)\in\mathrm O(r)$. 
Observe that both permutations and multiplications by the $2$-roots of unity $\pm 1$ make a subgroup of the orthogonal group. {We choose to discuss the $\mathrm{O}(r)$-orbits of decompositions, as they arise from a natural and well-understood action of $\mathrm{O}(r)$, rather than considering the quotient by this action.} 
This choice is also consistent with the notation utilized in \cite{FOST}. {In other words we consider any decompositions $f=\sum_{i=1}^r f_i^2$ as the element $(f_1,\dots,f_r)\in\prod_{i=1}^r\sym^{d/2}\C^{n+1}$ and not as an element in the Hilbert scheme of $r$ points in $\PP\sym^{d/2}\CC^{n+1}$.}

Although the identifiability of the decomposition in sum of squares is not possible for $r\geq 2$, due to the $\mathrm{O}(r)$-orbit,  we may relax this notion. We set accordingly the following definition.

\begin{Definition}
{An element $f$} with a minimal decomposition $
f=\sum_{i=1}^r f_i^2
$
is called {$\mathrm{O}(r)$-}identifiable if there is a unique $\mathrm{O}(r)$-orbit of decompositions of $f$, according to \eqref{eq:O(r)orbit}.
{We say $\sigma_r(\mathrm{Sq}_{d,n})$ is generically $\mathrm O(r)$-identifiable if the general element of $\sigma_r(\mathrm{Sq}_{d,n})$ is
{$\mathrm{O}(r)$-}identifiable.}
\end{Definition}

This orbit also impacts the expected dimension.  We have a dominant map
$$
\psi:\prod_{i=1}^r \sym^{d/2}\CC^{n+1}\to \sigma_r(\mathrm{Sq}_{d,n}),\ (f_1,\dots,f_r)\mapsto \sum_{i=1}^r f_i^2.
$$
The fibre of a generic point $f$ is $\psi^{-1}(f)=\{(f_1,\dots,f_r)\mid f=  \sum_{i=1}^r f_i^2\}$, it contains the $\mathrm{O}(r)$-orbit, so $\dim (\psi^{-1}(f))\geq \binom{r}{ 2}=\dim \mathrm{O}(r)$.

\begin{Definition}\label{def:expdim}
    The expected dimension of $\sigma_r(\mathrm{Sq}_{d,n})$ is 
    \begin{equation}\label{eq:edimsq}\mathrm{edim}(\sigma_r(\mathrm{Sq}_{d,n}))=r\binom{d/2+n}{n}-\binom{r}{2}.
    \end{equation}
    \end{Definition}
This notion coincides with the linear expected dimension defined in \cite[Definition 3.2]{BDP}. Note Definition \ref{def:expdim} differs from the standard expected dimension of a secant variety, where, in the subgeneric case, it is expected a unique decomposition and the term $\binom{r}{2}$ does not appear.
{Following \eqref{eq:edimsq} the expected generic $2$-rank is
\begin{equation}\label{eq:gen2rk}\min\left\{r\mid r\binom{d/2+n}{n}-\binom{r}{2}\ge \binom{d+n}{n}\right\},\end{equation} in agreement with \cite[Conjecture 1]{LSB} and \cite[Conjecture 1.2]{LORS}. In all cases that we have computed, the dimension of $\sigma_r(\mathrm{Sq}_{d,n})$ agrees with (\ref{eq:edimsq}) and the generic $2$-rank {agrees} with (\ref{eq:gen2rk}), see \Cref{prop: identM2} and \Cref{prop: dimM2}.

Moreover both (\ref{eq:edimsq}) and  (\ref{eq:gen2rk}) are realized in the case $d=2$, corresponding to the space $\sym^2\CC^{n+1}$ of symmetric matrices, that we identify as usual with quadratic forms, where the value in (\ref{eq:gen2rk}) is $n+1$. This case is traditionally considered defective for Waring decompositions, see \cite[Theorem 5.4.1.1]{LandsbergTensors}, \cite{BO07}, but in our $\mathrm{O}(r)$-setting it is no more defective, in the sense that the dimension of $\sigma_r(\mathrm{Sq}_{2,n})$ coincides with its expected dimension in (\ref{eq:edimsq})}. {Moreover, it is easy to check, essentially by the definition of $\mathrm{O}(r)$, that every $f\in\sigma_r(\mathrm{Sq}_{2,n})\setminus\sigma_{r-1}(\mathrm{Sq}_{2,n})$ is $\mathrm{O}(r)$-identifiable, in particular $\sigma_r(\mathrm{Sq}_{2,n})$ is generically $\mathrm{O}(r)$-identifiable}. 

{Indeed any $f\in\sigma_r(\mathrm{Sq}_{2,n})\setminus\sigma_{r-1}(\mathrm{Sq}_{2,n})$ is in the same $\mathrm{GL}(n+1)$-orbit of $\sum_{i=0}^{r-1}x_i^2$, and it is fixed by the subgroup $\mathrm O(n+1)\subset \mathrm {GL}(n+1)$, where the $\mathrm O(n+1)$-action on $\sum_{i=0}^{r-1}x_i^2$ coincides with the $\mathrm O(r)$-action in (\ref{eq:O(r)orbit}) by the embedding $\mathrm O(r)\subset \mathrm O(n+1)$ given by diagonal block matrix consisting of $\mathrm O(r)$ and the identity of size $n+1-r$. For other coordinate systems we get the same picture with a subgroup conjugate to the previous $\mathrm O(n+1)$. For $d\ge 4$ the picture is more complicated, indeed $r$ may be bigger than $n+1$ and $\mathrm O(r)$ cannot be embedded in $\mathrm O(n+1)$.} 

In \cite[Corollary 1.7]{FOST} it is shown that the dimension of $\sigma_r(\mathrm{Sq}_{d,n})$ is equal to the expected dimension for $r\leq n$. In particular, there exists a finite number of $\mathrm{O}(r)$-orbits of decomposition for a general polynomial of rank $r\le n$ \cite[Theorem 1.5]{FOST}. Moreover, for $r=2$ the orbit is unique \cite[Theorem 1.4]{FOST}. In this work, we give sufficient conditions for the $\mathrm O(r)$-identifiability utilizing the tangential contact locus, {by adapting a construction from \cite{CC}, which has been developed in \cite{COV14} and elsewhere}. {In particular, we show that the finitely many orbits of \cite[Theorem 1.5]{FOST} consist actually of a single orbit.} 

Next, we summarize the main results attained. In \Cref{sec: brk} we study {the case of binary forms.}

In \Cref{sec: ident} we give sufficient conditions for generic $\mathrm{O}(r)$-identifiability and we prove the following.

\begin{restatable*}{Theorem}{sufficientcondition}\label{thm:sufficientcondition}
     Let $f_i\in\sym^{d/2}\C^{n+1}$  be general. Let $H_1,\ldots, H_m$ be a basis of the space of hyperplanes vanishing on
$I_d=f_1\sym^{d/2}\C^{n+1}+\ldots+f_r\sym^{d/2}\C^{n+1}$.
In other words $\langle H_1,\ldots, H_m\rangle=I_d^\perp$ where $I=(f_1,\ldots, f_r)$.

Let $t_i$ be a basis of $\sym^{d/2}\C^{n+1}$, $N={{n+d/2}\choose n}$,
and
assume that the rank of the $N\times mN$ stacked Hessian matrix of scalars
$\frac{\partial H_1}{\partial t_i\partial t_j}|\ldots|\frac{\partial H_m}{\partial t_i\partial t_j}$ is equal to $N-r$. Then
 $\sigma_r(\mathrm{Sq}_{d,n})$ has dimension $rN-{r\choose 2}$ and it is generically $\mathrm O(r)$-identifiable.
\end{restatable*}

In the above statement, the perpendicular $\perp$ is considered, as usual, with respect to the apolar product (see \Cref{sec:dual}).
This leads to a criterion for specific $\mathrm{O}(r)$-identifiability, see \Cref{thm:sufficientconditionspecific}. Furthermore, we utilise \Cref{thm:sufficientconditionspecific} (2) to verify generic $\mathrm{O}(r)$-identifiability in Macaulay2 \cite{M2} for several cases, these are collected in \Cref{prop: identM2}.

\begin{restatable*}{Proposition}{Fermat}\label{prop:Fermat}
 Assume $n\ge 2$, $d\ge 2+\frac{4}{n-1}$, $r\le n+1$. Then
$\sigma_r(\mathrm{Sq}_{d,n})$ is generically $\mathrm O(r)$-identifiable.
\end{restatable*}

This shows the uniqueness of the orbit in \cite[Theorem 1.5]{FOST}. In particular, this allows us to conclude the generic $\mathrm{O}(r)$-identifiability of ternary forms of subgeneric rank (\Cref{cor: ternary}). In Corollary \ref{Cor: BOP2} we prove a sufficient condition for non defectivity which uses the work by Brambilla, Dumitrescu and Postinghel \cite{BDP}. In \S \ref{sec:dual} we describe the dual variety to $\sigma_{r}(\mathrm{Sq}_{d,n})$. In \S \ref{sec:computational} we report some computational results obtained with Macaulay2 \cite{M2}.

{ \subsection*{Acknowledgement}
We thank Elisa Postinghel who explained us how from \cite[Theorem 5.3]{BDP} it follows \Cref{Thm: BOP5.3}.
{{We thank Daniel Plaumann who found a mistake in the first version of this paper and kindly pointed out the paper \cite{Wall} by Wall, where ternary quartics of $2$-rank equal to $4$ are classified.}}
The first author is member of GNSAGA of INDAM and was partially supported by PRIN project "Multilinear Algebraic Geometry" of MUR. The second author was partially supported by the project Pure Mathematics in Norway, funded by Trond Mohn Foundation and Tromsø Research Foundation. We are grateful to the referee for a careful reading and the useful suggestions.}

\section{{Binary forms}}\label{sec: brk}

We give a precise description of the $\mathrm{O}(2)$-orbits of decompositions for generic binary forms of any degree $d$. 

\begin{Proposition}\label{prop: binary}
The number of $\mathrm{O}(2)$-orbits of decompositions of a general binary form of degree $d$
as a sum of two {squares} is equal to ${{d-1}\choose d/2}={\frac 12}{{d}\choose d/2}$.
\end{Proposition}
\begin{proof}
    Let $f=\prod_{i=1}^d l_i$ and $A,B$ be complementary subsets of $\{1,\dots,d\}$ of cardinality $d/2$, then $$f=\left(\frac 12\left(\prod_{i\in A}l_{i}+\prod_{j\in B}l_{j}\right)\right)^2+\left(\frac i2\left(\prod_{i\in A}l_{i}-\prod_{j\in B}l_{j}\right)\right)^2.$$

There are ${{d}\choose d/2}$ choices for $A,B$, and we show that each choice leads to a different $\mathrm{O}(2)$-orbit. Up to  scalar multiplication, we may assume $f=\prod^d_{i=1}(x+a_iy)$, let $A=\{1,\dots,d/2\}, B=\{d/2+1,\dots,d\}$. Let $$C=\begin{bmatrix}
2&\sum a_i& \sum_{i< j\in A}a_ia_j+\sum_{i< j\in B}a_ia_j&\dots&\prod_{i\in A}a_i+\prod_{i\in B}a_i\\
0&\sum_{i\in A} a_i-\sum_{i\in B}a_i& \sum_{i< j\in A}a_ia_j-\sum_{i< j\in B}a_ia_j&\dots&\prod_{i\in A}a_i-\prod_{i\in B}a_i
\end{bmatrix}$$ be the matrix which has the coefficients of $\prod_{i\in A}(x+a_iy)+\prod_{i\in B}(x+a_iy)$ in the first row and the coefficients of $\prod_{i\in A}(x+a_iy)-\prod_{i\in B}(x+a_iy)$ in the second row. {Note that $\mathrm O(2)$ acts by left multiplication on $C$, so that} in each $\mathrm O(2)$-orbit we may find such a decomposition when the first column of $C$ is $\begin{bmatrix}
2\\0\end{bmatrix}$. Let $\sigma \in \mathfrak S_d$ be a permutation, $A'=\sigma A, B'=\sigma B$, and let $C'$ be the matrix of coefficients of this new decomposition. If these two choices of $A, B$, $A',B'$ are in the same $\mathrm O(2)$-orbit, there exists $M=\begin{bmatrix}
    \alpha&\beta\\
    \gamma&\delta
\end{bmatrix}\in O(2)$ such that $MC=C'$. Looking at the first column, we get $\alpha=1$, and since $\alpha^2+\beta^2=1$,  $\beta=0$. Moreover $\gamma=0$, this implies $\delta=\pm 1$. Note that $\delta=-1$ corresponds to swapping $A$ and $B$.\qedhere

\end{proof}

\Cref{prop: binary} is {essentially} known from \cite[Theorem 5]{FOS}. We just formulate it to make the orbit structure of the decomposition explicit and show they are all distinct.  \Cref{prop: binary} also gives examples of special forms of subgeneric $2$-rank $r$ that have more than one $\mathrm{O}(r)$-orbit of decompositions, it is enough to consider a binary form as a special ternary form.

\section{Identifiability of sum of squares}\label{sec: ident}

The description of tangent spaces to secant varieties is a classical result by Terracini proved in \cite{terracini}. {We restate it in the affine setting.}

\begin{Lemma}[Terracini's Lemma]\label{lemma:Terracini}
Let $X$ be the affine cone over a nondegenerate and irreducible projective variety. Consider generic points $x_1, \dots, x_r\in X$ and {$x = \sum_{i=1}^rx_i$}. Then $$
T_x\sigma_r(X)=\langle T_{x_1}X,\dots,T_{x_r}X \rangle.
$$

\end{Lemma}

{Note the tangent space to $\mathrm{Sq}_{d,n}$ at $f^2\in \mathrm{Sq}_{d,n}$ is 
$T_{f^2}\mathrm{Sq}_{d,n}=\left\{fg\mid g\in \sym^{d/2}\C^{n+1}\right\}$. The $r$-tangential contact locus was used in \cite{COV14} to give an identifiability criterion for tensor decompositions.
Here we use it to give a $\mathrm O(r)$-identifiability criterion for sums of squares, in \Cref{thm:sufficientcondition}
.}

\begin{Definition}
    Let $x_1,\dots,x_r\in X$ and $M=\langle T_{x_1}X,\dots,T_{x_r}X \rangle$. We define the $r$-tangential contact locus $\mathcal{C}_r(x_1,\dots,x_r)$ on $X$ by $$
    \mathcal{C}_r(x_1,\dots,x_r)=\{x\in X\ |\ T_xX\subset M \}.
    $$
\end{Definition}

\begin{Proposition}\label{prop:contactlinear} 
Let $f_1^2,\dots,f_r^2\in \mathrm{Sq}_{d,n}$.
\begin{enumerate}
\item{} The $r$-tangential locus
$\mathcal{C}_r(f_1^2,\dots,f_r^2)$ of $\mathrm{Sq}_{d,n}$ is a linear space in $\PP\sym^{d/2}\C^{n+1}\simeq \mathrm{Sq}_{d,n}$.

\item{} Let $N={{n+d/2}\choose n}$, $I=(f_1,\dots,f_r)$, $I_d=f_1\sym^{d/2}\C^{n+1}+\ldots+f_r\sym^{d/2}\C^{n+1}$, $\langle H_1,\ldots, H_m\rangle=I_d^\perp$ a basis of hyperplanes vanishing on $I_d$.
The dimension of $\mathcal{C}_r(f_1^2,\dots,f_r^2)$ is given by 
$N$ minus the rank of the $N\times mN$ stacked Hessian matrix of scalars
$\frac{\partial H_1}{\partial t_i\partial t_j}|\ldots|\frac{\partial H_m}{\partial t_i\partial t_j}$,
where $t_i$ is a basis of $\sym^{d/2}\C^{n+1}$.

\item{} Let $\{f_1,\ldots, f_r,s_{r+1},\ldots, s_N\}$ be a basis of $\sym^{d/2}\C^{n+1}$.
The dimension of $\mathcal{C}_r(f_1^2,\dots,f_r^2)$ is given by 
$N$ minus the rank of the $(N-r)\times m(N-r)$ stacked Hessian matrix
of scalars $\frac{\partial H_1}{\partial s_i\partial s_j}|\ldots|\frac{\partial H_m}{\partial s_i\partial s_j}$.
\end{enumerate}
\end{Proposition}
\begin{proof}
We have that $g\in \PP\sym^{d/2}\C^{n+1}$ belongs to $\mathcal{C}_r(f_1^2,\dots,f_r^2)$ if and only if
$g\sym^{d/2}\C^{n+1}$ is contained in $f_1\sym^{d/2}\C^{n+1}+\ldots+f_r\sym^{d/2}\C^{n+1}$.
Given $H_p$ and $t_i$ as in (2), this amounts to say that
$\frac{\partial H_p}{\partial t_i}(g)=0$ $\forall p=1,\ldots m$,
$\forall t_i$, which gives a linear system for the unknown $g$. {The matrix of this linear system with the $N$ coefficients of $g$ as unknowns is the stacked Hessian matrix appearing in (2).} This proves (1),
and considering the rank of the linear system proves also (2).
Item (3) is proved since each Hessian matrix $\frac{\partial H_p}{\partial t_i\partial t_j}$ in (2) , when computed with respect to the basis of item (3), is a symmetric matrix with the first $r$ rows vanishing, since $H_p$ vanishes on $f_i\sym^{d/2}\C^{n+1}$.
\end{proof}

\begin{Lemma}
   Let $f_i\in\sym^{d/2}\C^{n+1}$ for $i=1,\ldots, r$ and let $f=f_1^2+\ldots+f_r^2$ be a minimal decomposition of $f$. Then $f_i$ are linearly independent.
\end{Lemma}
\begin{proof}
    Assume $f_r$ is a linear combination of $f_1,\ldots, f_{r-1}$. Then $f$ is a quadratic form in the variables $f_1,\ldots, f_{r-1}$ and can be written as a sum of at most $r-1$ squares of linear forms in these variables.\qedhere
\end{proof}

\begin{Proposition}\label{proporto} 
Let $N={d/2+n\choose n}$. Assume the tangent spaces at $f_1^2,\ldots, f_r^2$ have span of maximal possible affine dimension $rN-{r\choose 2}$. Then $f_if_j$ for $1\le i\le j\le r$ are linearly independent. In particular if $\sum_{i=1}^rf_i^2=\sum_{i=1}^r(\sum_{j=1}^rm_{ij}f_j)^2$,
then $(m_{ij})$ is an orthogonal matrix.
\end{Proposition}
\begin{proof}
Let $\sym^{d/2}\C^{n+1}=\langle f_1,\ldots, f_r,s_{r+1},\ldots, s_{N}\rangle$.
Denote $S=\langle s_{r+1}\ldots, s_{N}\rangle$.
We have $T_{f_1^2}\mathrm{Sq}_{d,n}+\ldots+T_{f_r^2}\mathrm{Sq}_{d,n}=\langle f_if_j\rangle
    +f_1S+\ldots +f_rS$,
    so that $$rN-{r\choose 2}\le \dim\langle f_if_j\rangle+r\left(N-r\right)$$
hence $\dim\langle f_if_j\rangle\ge{{r+1}\choose 2 }$. The second claim is now straightforward.
\end{proof}

\def\niente{
\begin{Example}
Let $f=x_0^8+x_1^8+x_2^8$, then $\sym^4\CC^3=\langle x_0^4,x_1^4,x_2^4,R \rangle$, where $R=\{x^\alpha| \alpha_i<4 \text{ for every }i\}$. Then $R^2=\langle x^\alpha | \alpha_i<8 \text{ for every }i\rangle$, while $T_f\mathrm{Sq}_{8,3}=x_0^4\sym^4\CC^3+x_1^4\sym^4\CC^3+x_2^2\sym^4\CC^3=\langle x^\alpha | \alpha_i\geq 4 \text{ for at least one }i\rangle$. This means $\mathcal R=\langle x^\alpha | \alpha_i\leq 3 \text{ for every }i\rangle=\langle x^3y^3z^2,x^3y^2z^3,x^2y^3z^3 \rangle$. Therefore the system of equations \Cref{proptcl} is given by the $\sum_j d_jr_ir_j$ such that $r_ir_j\in\{x^3y^3z^2,x^3y^2z^3,x^2y^3z^3\}${\blue Why ?, we have $\sum_j d_jr_ir_j\in T_f\mathrm{Sq}_{8,3}$. So we have to contract the matrix $(r_ir_j)$ by $(T_f\mathrm{Sq}_{8,3})^\perp$. For example contract $(r_ir_j)$ by $x^3y^3z^2$, does it have maximal rank $12$? No, it has rank $10$, only contracting by linear combination $x^3y^3z^2+x^3y^2z^3+x^2y^3z^3$ we get $12$.}. Fixing $i$ gives an equation in the $d_j$'s such that $r_ir_j\in \{x^3y^3z^2,x^3y^2z^3,x^2y^3z^3\}$, and the opposite holds, $d_i$ will be a factor in the equation given by fixing each of the previous $j$'s, we get a square system, that is independent by \Cref{lem:indep}, thus $0$ is the only solution. This can be generalized to any degree and number of variables utilizing the Fermat form.
\end{Example}

\begin{Example}[General idea]
    {\blue To be written better if you agree that works.} Let $\sym^{d/2}\CC^{n+1}=\langle f_1,\dots,f_k,r_{k+1},\dots,r_N\rangle$ be a basis. Then $\mathcal R_d=\left(R^2/(T_f\mathrm{Sq}_{d,n}\cap R^2)\right)_d$ is a linear space of dimension $c=\binom{d+n}{n}+\binom{k}{2}-k\binom{d/2+n}{n}$, because $R^2+T_f\mathrm{Sq}_{d,n}=\sym^d\CC^{n+1}$, and $\dim T_f\mathrm{Sq}_{d,n}=k\binom{d/2+n}{n}-\binom{k}{2}$. Denote $\ell_1,\dots,\ell_c$ a basis of $\mathcal R$, then $[r_ir_j]=\sum_{t=1}^c\alpha_t^{(ij)}\ell_t$. Therefore the projection of $\sum_{j=k+1}^Nd_jr_jr_i$ to $\mathcal R/\mathcal{R}\cap (T_{f_1}+\dots+T_{f_k})$ is $$\sum_{t=1}^c\sum_{j=k+1}^Nd_j\alpha_t^{(ij)}\ell_t=0.$$
    Since $\ell_t$ are independent, we have $$
    \sum_{j=k+1}^N d_j\alpha_t^{(ij)}=0
    $$
    for every $i\in\{k+1,\dots,N\}$ and $t\in\{1,\dots,c\}$. Fixing $t$, this leads to a linear system on the $d_j's$ described by a symmetric $(N-(k+1))\times (N-(k+1))$-matrix $$
    M_t=(\alpha_t^{(ij)})_{ij}.
    $$
    We know that for each $i$, there exists at least one $j,t$ such that $\alpha_{t}^{i,j}=\alpha_{t}^{j,i}\neq 0$. Thus a general linear combination $\sum \beta_tM_t=M$ is a matrix with all entries non-zero. Moreover, $(\alpha_1^{i,j},\dots,\alpha_c^{i,j})\neq \lambda (\alpha_1^{s,t},\dots,\alpha_c^{s,t})$ if $(s,t)\neq (i,j),(j,i)$.

    Can we guarantee that a random linear combination of $M_T$ has maximal rank from this? Maybe some contradiction arrives from assuming this is not true, I cannot see it though.
\end{Example}

\begin{Lemma}\label{lem:indep}
Let $r_i$ be independent forms for $i=1,\ldots p$ in $\sym^{d/2}V$.
Then the $p\times p$ matrix $(r_ir_j)$ with entries in $\sym^{d}V$
has maximal rank $p$ when contracted by a general element of $\sym^{d}V^\vee$. 
\end{Lemma}
\begin{proof}
Fix a monomial order.    We may assume $r_i$ in reduced form,
that is $LT(r_i)=x^{\alpha_i}$, $x^{\alpha_1}>x^{\alpha_2}>\ldots>x^{\alpha_p}$
and the coefficient of $x^{\alpha_i}$ vanishing in $r_j$ for $i\neq j$.
Contract the matrix $(r_ir_j)$ {\blue Which matrix? I do not understand also which determinant afterwards} with respect to $\sum_{i=1}^p c_i x^{2\alpha_i}$.
Then in the development of $\det (r_ir_j)$  there is a diagonal term
$c_1c_2\ldots c_p$ that cannot cancel since $c_1$ appears only in entry $(1,1)$,
in the complementary minor $c_2$ appears only in entry $(2,2)$ and so on.
So the determinant is a nonzero polynomial in $c_i$ which does not vanish for general $c_i$.
\end{proof}

\identifiability

\begin{proof}
We consider two decompositions in different $\mathrm{O}(r)$-orbits
$$ \sum_{i=1}^rf_i^2 =\sum_{i=1}^rg_i^2 $$

By Terracini's \Cref{lemma:Terracini}
$$T_{f_1}+\ldots+T_{f_r} = T_{g_1}+\ldots+T_{g_r}.$$
Thanks to \Cref{proporto} we may assume that $\langle f_1,\ldots, f_r\rangle\neq\langle g_1,\ldots, g_r\rangle$,
otherwise the two decompositions lie in the same orbit. 
From \Cref{proptcl} we get
$$ \{(\sum_{i=1}^r c_if_i)^2|c_i\in\CC\} = \{(\sum_{i=1}^r c_ig_i)^2|c_i\in\CC\},$$
hence
 $\langle f_1,\ldots, f_r\rangle = \langle g_1,\ldots, g_r\rangle$
 which is a contradiction.
\end{proof}
}

The following Lemma is well known and we omit its straightforward proof.

\begin{Lemma}\label{lem:isotropic}
    Let $i\in\{1,\ldots r\}$.
    The set $\{\sum_{j=1}^rm_{ij}f_j\mid(m_{ij})\in O(r)\}$
    is equal to 
    $\{\sum_{\ell=1}^r\lambda_\ell f_\ell\mid\sum_{\ell=1}^r\lambda_\ell^2=1\}$
    and its projectivization is dense in $\PP\langle f_1,\ldots , f_r\rangle$.
\end{Lemma}

We now establish the connection between identifiability and the tangential contact locus.

\begin{Proposition}\label{prop: nonidentifiability} 
    Suppose $\sigma_r(\mathrm{Sq}_{d,n})$ is not generically $\mathrm{O}(r)$-identifiable, $f_1^2,\dots,f_r^2\in \mathrm{Sq}_{d,n}$ are generic and $\langle T_{f_1^2}\mathrm{Sq}_{d,n},\dots,T_{f_r^2}\mathrm{Sq}_{d,n}\rangle$ has dimension $r\binom{n+d/2}{d/2}-\binom{r}{2}$ , then the tangential contact locus $\mathcal C_r(f_1^2,\dots,f_r^2)$ contains a variety of affine dimension $\geq r+1$.
\end{Proposition}
\begin{proof}

    Suppose $f=\sum_{i=1}^r (\sum_{j=1}^rm_{ij}f_j)^2=\sum_{i=1}^r (\sum_{j=1}^r n_{ij}g_j)^2$ are two different orbits of decomposition for $f$, where $(m_{ij}),(n_{ij})\in \mathrm{O}(r)$.
    Since the orbits are different there is at least one $g_i\notin\langle f_1,\dots,f_r\rangle$. Indeed, if for every $i=1,\dots,r$ we would have $g_i\in\langle f_1,\dots,f_r\rangle$, then $g_i^2=(\sum_{j=1}^r\alpha_{ij}f_j)^2$, thus $$
    \sum_{i=1}^r g_i^2=\sum_{i=1}^r (\sum_{j=1}^r \alpha_{ij}f_j)^2=\sum_{i=1}^r f_i^2.
    $$
    If $\sigma_r(\mathrm{Sq}_{d,n})$ has the expected dimension $r\binom{n+d/2}{d/2}-\binom{r}{2}$ then the tangent spaces at $f_i^2$ have span of maximal dimension. Therefore we satisfy the assumptions of \Cref{proporto}, and it follows that $(\alpha_{ij})\in\mathrm{O}(r)$, thus both decompositions lies in the same orbit, a contradiction. This means that the tangential contact locus of $f_1^2,\ldots, f_r^2$, which by \Cref{prop:contactlinear} is a linear space containing { $\langle f_1,\dots,f_r\rangle$, contains also $g_i\notin \langle f_1,\dots,f_r\rangle $ by Terracini Lemma, then its dimension is $\ge r+1$.}
\end{proof}

We are ready to prove our criterion for generic identifiability.

\sufficientcondition

\begin{proof}
    The assumptions imply that 
    $\mathcal{C}_r(f_1^2,\dots,f_r^2)$ for general $f_i$ is a linear space of affine dimension $r$
    by \Cref{prop:contactlinear} (2).
    Then the result follows from the contrapositive of \Cref{prop: nonidentifiability}.
\end{proof}

The following is a sufficient criterion for $\mathrm O(r)$-identifiability of a \emph{specific} polynomial $f$ .

\begin{Theorem}\label{thm:sufficientconditionspecific}
    Let $N={d/2+n\choose n}$, $f\in\sym^d\C^{n+1}$, \begin{equation}
    \label{eq:fsosr}f=\sum_{i=1}^r f_i^2,\end{equation} 
    assume it is a smooth point of $\sigma_r(\mathrm{Sq}_{d,n})$, which we assume to have dimension equal to the expected dimension $rN-\binom{r}{2}$.
    Let $I=(f_1,\dots,f_r)$, $\langle H_1,\ldots, H_m\rangle=I_d^{\perp}$ be a basis of hyperplanes vanishing on
$I_d=f_1\sym^{d/2}\C^{n+1}+\ldots+f_r\sym^{d/2}\C^{n+1}$. Assume that $\dim I_d= rN-\binom{r}{2}$ and let $t_i$ be a basis of $\sym^{d/2}\C^{n+1}$.
\begin{enumerate}
\item{}
Assume that the rank of the $N\times mN$ stacked Hessian matrix of scalars
$\frac{\partial H_1}{\partial t_i\partial t_j}|\ldots|\frac{\partial H_m}{\partial t_i\partial t_j}$ is equal to $N-r$. Then $f$ is $\mathrm O(r)$-identifiable, in other words, all decomposition of $f$ as a sum of $r$ squares are in the same $\mathrm O(r)$-orbit of 
(\ref{eq:fsosr}).

\item{} Let $\{f_1,\ldots, f_r,s_{r+1},\ldots, s_N\}$ be a basis of $\sym^{d/2}\C^{n+1}$.
Assume that the rank of the $(N-r)\times m(N-r)$ stacked Hessian matrix
of scalars $\frac{\partial H_1}{\partial s_i\partial s_j}|\ldots|\frac{\partial H_m}{\partial s_i\partial s_j}$ is maximum, so equal to $N-r$. Then
$f$ is $\mathrm O(r)$-identifiable, in other words all decomposition of $f$ as sum of $r$ squares are in the same $\mathrm O(r)$-orbit of 
(\ref{eq:fsosr}).
\end{enumerate}
\end{Theorem}

\begin{proof}
    The assumptions (1) (respectively (2)) imply that 
    $\mathcal{C}_r(f_1^2,\dots,f_r^2)$ is a linear space of affine dimension $r$
    by \Cref{prop:contactlinear} (2) (respectively (3)).
    Then the result follows from \Cref{prop: nonidentifiability} and a modification of the arguments in the proof in \cite[Prop. 5.1]{COV}, as follows.

     Let $f=\sum_{i=1}^r f_i^2$ be a smooth point of $\sigma_r(\mathrm{Sq}_{d,n})$ and $\dim (\mathcal C_r(f_1^2,\dots,f_r^2))=r$. Then there exists an open neighbourhood of $f$ where every point is smooth, its contact locus is $r$-dimensional and its tangent space is described as in Terracini's Lemma. In particular, this implies $\sigma_r(\mathrm{Sq}_{d,n})$ is generically $\mathrm O(r)$-identifiable by \Cref{thm:sufficientcondition}.

    Let $$\mathcal A(\sigma_r(\mathrm{Sq}_{d,n}^\PP))=\left\{\left([g],\left([g_1^2],\dots,[g_r^2]\right)\right)\mid g\in\langle g_1^2,\dots,g_r^2\rangle \right\}\subset\PP\sym^d\CC^{n+1}\times \prod_{i=1}^r\mathrm{Sq}_{d,n}^\PP$$
    be the abstract $r$-secant variety of squares, and $\pi:\mathcal A(\sigma_r(\mathrm{Sq}_{d,n}^\PP))\to \PP\sym^d\CC^{n+1}$ the projection to the first factor. Notice $\pi(\mathcal A(\sigma_r(\mathrm{Sq}_{d,n}^\PP)))=\sigma_r(\mathrm{Sq}_{d,n}^\PP)$ and the generic fibre consists of the unique $\mathrm O(r)$-orbit of decompositions, in particular it has dimension $r\choose 2$.

    Assume $\pi^{-1}([f])$ contains two different $\mathrm O(r)$-orbits of decompositions, i.e., there exist two points $\left([f],\left([f_1^2],\dots,[f_r^2]\right)\right)$, $\left([f],\left([h_1^2],\dots,[h_r^2]\right)\right)\in \pi^{-1}([f])$ such that $h_i\notin\langle f_1,\dots,f_r\rangle$ for at least one index $i$, {otherwise the orbit would be the same as explained previously in the proof of \Cref{prop: nonidentifiability}}. Terracini's Lemma and its proof imply that the {tangent space at $f=\sum_{i=1}^r f_i^2$
    of $\sigma_r(\mathrm{Sq}_{d,n})$ contains the span of    $T_{f_i^2}\mathrm{Sq}_{d,n}$ for $i=1,\ldots, r$.
    This span is exactly the space $I_d$ in the statement. Since the dimension of $I_d$ coincides by assumption with the dimension of $\sigma_r(\mathrm{Sq}_{d,n})$, which is smooth at $f$,  we have that $I_d$ coincides with the tangent space at $f$.
    It follows that the derivative of $\pi$ 
    (whose Jacobian matrix is sometimes called the Terracini matrix) drops rank at $([f],([f_1^2],\dots,[f_r^2]))$ exactly by $r\choose 2$,
hence the connected component of the fibre containing $([f],([f_1^2],\dots,[f_r^2]))$ cannot have dimension larger than $r\choose 2$ and it coincides with the $\mathrm O(r)$-orbit.} The $\mathrm O(r)$-orbit containing $\left([f],\left([h_1^2],\dots,[h_r^2]\right)\right)$ must lie in another connected component of $ \pi^{-1}([f])$.

    However $[f]$ is a smooth point of $\sigma_r(\mathrm{Sq}_{d,n}^\PP)$ and $\pi$ is a surjective proper morphism, thus Zariski Connectedness Theorem implies that $\pi^{-1}([f])$ is connected, contradicting the previous paragraph. Therefore, $f$ is $\mathrm{O}(r)$-identifiable.\qedhere
    
\end{proof}

\Fermat

\begin{proof}
The first inequality in the assumption 
is equivalent to $(n+1)(d/2-1)\ge d$, so under this assumption we have a monomial $x^{\alpha}$  of degree $d$ with $\alpha_i\le (d/2-1)$ for $i=0,\ldots, n$.
{We may assume $r=n+1$.} 

Our strategy is the following. We will show that $\mathcal C(x_0^d,\dots,x_n^d)=\{(\sum_{i=0}^{n}c_ix_i^{d/2})^2\mid c_i\in \CC\}$, so in particular it has dimension $n+1$. {Then we may apply Proposition \ref{prop: nonidentifiability} to conclude identifiability. }

Let $N=\binom{n+d/2}{n}$, 
and denote $R=\{x^\alpha\mid |\alpha|=\frac{d}{2}, \alpha_i<\frac{d}{2}\ \forall \ i\in\{0,\dots,n\}\}=\{r_{n+2},\dots,r_N\},$ {so that $\dim R=N-(n+1)$. 
Consider 
$$\sum_{i=0}^n T_{x_i^d}\mathrm{Sq}_{d,n}=\sum_{0\leq i\leq j\leq n}\langle x_i^{d/2}x_j^{d/2}\rangle+\sum_{i=0}^n x_i^{d/2}R,$$
hence we have $\dim \sum_{i=0}^n T_{x_i^d}\mathrm{Sq}_{d,n} ={{n+2}\choose 2}+(n+1)\left(N-(n+1)\right)= (n+1)N-{{n+1}\choose 2}$, in agreement with the assumptions of \Cref{prop: nonidentifiability}.}
 
{
Moreover $R^2 + \sum_{i=0}^n T_{x_i^d}\mathrm{Sq}_{d,n}=S^d\CC^{n+1}${, in the following we will analyze in detail this sum, which, in particular, is not a direct sum}. Let $$\mathcal R= (R^2 + \sum_{i=0}^n T_{x_i^d}\mathrm{Sq}_{d,n})/\sum_{i=0}^n T_{x_i^d}\mathrm{Sq}_{d,n}\simeq R^2/(R^2\cap\sum_{i=0}^n T_{x_i^d}\mathrm{Sq}_{d,n}).$$

It follows $c:=\dim(\mathcal R)=\binom{n+d}{n}+\binom{n+1}{2}-N(n+1)$, and let $\{[x^\beta]\mid |\beta|=d,\beta_i<\frac{d}{2}\}=\{s_1,\dots,s_c\}$ be a basis of $\mathcal R$, {seen as representatives of lateral classes in $S^d\CC^{n+1}$ modulo $\sum_{i=0}^n T_{x_i^d}\mathrm{Sq}_{d,n}$. }

{We show now that $\mathcal C(x_0^d,\dots,x_n^d)=\{(\sum_{i=0}^{n}c_ix_i^{d/2})^2\mid c_i\in \CC\}$.} Suppose that $(\sum_{i=0}^nc_ix_i^{d/2}+\sum_{j=n+2}^N d_jr_j)^2\in \mathcal C(x_0^d,\dots,x_n^d)$, $c_i,d_j\in \CC$, we have to prove that $d_j=0$ for $j=n+2,\ldots, N$. From the definition of tangential contact locus we have that {$(\sum_{i=0}^nc_ix_i^{d/2}+\sum_{j=n+2}^N d_jr_j)r_\ell\in \sum_{i=0}^n T_{x_i^d}\mathrm{Sq}_{d,n}
$.} It follows that its projection to $\mathcal R$ vanishes, so $[\sum_{j=n+2}^N d_jr_jr_\ell]=0\in \mathcal R$ for every $\ell=n+2,\dots, N$. 

Since the $r_j=x^{\alpha}$ are the elements of the monomial basis such that $\alpha_i<\frac{d}{2}$ for every $i=0,\dots,n$, we have $r_ir_\ell=r_jr_\ell$ if and only if $i=j$, so if $[r_ir_\ell]\neq0$, then $[r_ir_\ell]=[r_jr_\ell]$ if and only if $i=j$. Moreover, from the definition of the $s_i$, we have that $s_i=[r_kr_\ell]$ for some $k$ and $\ell$ . Therefore for each $\ell=n+2,\dots, N$ {there exists a subset $J_\ell\subset
\{n+2,\dots,N\}$ such that} we have $$
\left[\sum^N_{j=n+2}d_jr_jr_\ell\right]=\sum_{j\in J_\ell \subset\{n+2,\dots,N\}} d_j[r_jr_\ell]=0,
$$
where $[r_jr_\ell]\neq 0$ for $j\in J_\ell$.
Since $\{[r_jr_\ell]\mid j\in J_\ell\}\subset\{s_1,\dots,s_c\}$ and $[r_jr_\ell]\neq[r_ir_\ell]$ for $i\neq j\in J_\ell$, it follows that it is a set of linearly independent vectors in $\mathcal R$, so $d_j=0$ for all $j\in J_\ell$. Moreover, notice that for each fixed $j\in\{n+2,\dots,N\}$, there exists at least one $\ell\in \{n+2,\dots,N\}$ such that $[r_jr_\ell]\neq0$, therefore $d_j=0$ for all $j\in\{n+2,\dots,N\}$.
{This concludes the proof that $\mathcal C(x_0^d,\dots,x_n^d)=\{(\sum_{i=0}^{n}c_ix_i^{d/2})^2\}$. It follows by semicontinuity that $\dim (\mathcal C(f_1^2,\dots,f_n^2))=r$ for generic $f_1,\dots,f_r$
and applying Proposition \ref{prop: nonidentifiability} we get the desired $\mathrm O(r)$-identifiability.}
}\end{proof}

\begin{Remark}
The technical condition {$d\ge 2+\frac{4}{n-1}$} is satisfied for all $n\geq 3$ and $d\geq4$. For $n=2$ it is not satisfied only for $d=4$, in which case the generic rank is $3=n+1$. However, in such case it is known by \cite[Theorem 1.4]{FOST} that the $2$-secant is generically identifiable, that is the only non-trivial subgeneric case. For $d=2$, $\sigma_n(\mathrm{Sq}_{2,n})$ is the determinant hypersurface and $\sigma_{n+1}(\mathrm{Sq}_{2,n})\setminus \sigma_{n}(\mathrm{Sq}_{2,n})$ consists of all symmetric matrices of maximal rank, every element in this $\mathrm{SO}(n+1)$-orbit  is easily seen to be $\mathrm{SO}(n+1)$-identifiable by the definition of orthogonal group. In the same way, every element 
of $\sigma_{j}(\mathrm{Sq}_{2,n})\setminus \sigma_{j-1}(\mathrm{Sq}_{2,n})$ is a symmetric matrix of rank $j$ and it is $\mathrm{SO}(j)$-identifiable; here the contact locus can be identified with the span of the columns of the matrix.
\end{Remark}

\begin{Corollary}\label{cor: ternary}
    Ternary forms of subgeneric rank are generically $\mathrm{O}(r)$-identifiable.
\end{Corollary}
\begin{proof}
    The generic rank of squares in $\sym^d\CC^{3}$ is at most $4$ \cite[Theorem 4]{FOS}. So we have $r\le 3$ and the result follows then by \Cref{prop:Fermat}.
\end{proof}

\begin{Remark}
    \Cref{cor: ternary} also implies that $\sigma_r(\mathrm{Sq}_{d,2})$ has dimension equal to its expected dimension, however, it is important to mention this was previously noted from a known case of Fr\"oberg's Conjecture \cite{Fro}.  

    In a nutshell, for a homogeneous ideal $I=(f_1,\dots,f_r)$ generated by generic forms, Fr\"oberg's Conjecture foresees the dimension of the degree $d$ piece $I_d$ of the ideal $I$. In our setting, $\deg(f_1)=\dots=\deg(f_r)=d/2$, then $I_d=T_f{\sigma_r(\mathrm{Sq}_{d,n})}$, for $f=\sum^r_{i=1} f_i^2$, so $I_d$ corresponds to the tangent plane in a generic point $f$. The connections of Fr\"oberg's Conjecture and the expected dimension of secant varieties is described in details in \cite{Oneto2016thesis}. {The main point is that if Fr\"oberg's Conjecture is true then the secant varieties ${\sigma_r(\mathrm{Sq}_{d,n})}$ are never defective. }

    In the particular case of \Cref{cor: ternary}, we have $n=2$ and $r\leq 3$, in this case Fr\"oberg's Conjecture has been proved, as shown in \cite{Sta}, see \cite[\S 3, Example 2]{Fro}, it holds for $r\leq n+2$.
\end{Remark}
To verify non-defectiveness we recall \cite[Theorem 5.3]{BDP}.

\begin{Theorem}[{{\cite[Theorem 5.3]{BDP}}}]\label{Thm: BOP5.3}
Let $b=\min\{n,r-n-2\}$, then $\sigma_r(\mathrm{Sq}_{d,n})$ is non-defective if $$
r\left(\frac{d}{2}+1\right)\leq nd+b.
$$    
\end{Theorem} 
\begin{proof}
   We specialize $r$ general forms of degree $d/2$ to $l_i^{d/2}$ for $i=1\ldots r$ where $\deg l_i=1$. The span of tangent spaces to $\mathrm{Sq}_{d,n}$ at $sq(l_i^{d/2})=l_i^d$ 
   is the degree $d$ piece of the ideal $(l_1^{d/2},\ldots, l_r^{d/2})$, which is dual to the linear system of degree $d$ forms
   having multiplicities $d/2+1$ at the points dual to $l_i$. This is denoted as ${\mathcal L}_{n,d}\left((d/2+1),\ldots,(d/2+1)\right)$ in \cite{BDP}. The result follows from \cite[Theorem 5.3]{BDP} since from their notations we have $s(d)=0$, i.e., there are no points of multiplicity $d$.
\end{proof}
\begin{Corollary}\label{Cor: BOP1}
$\sigma_r(\mathrm{Sq}_{d,n})$ is non-defective for \begin{enumerate}
    \item $n=4,\ r=7,\ d\geq 12$.
    \item $n=5,\ r=8,\ d\geq 7$.
    \item $n=5,\ r=9,\  d\geq 14$.
    \item $n=6,\ r=9,\  d\geq 6$.
    \item $n=6,\ r=10,\  d\geq 8$.
    \item $n=6,\ r=11,\  d\geq 16$.
\end{enumerate}   
\end{Corollary}
\begin{Corollary}\label{Cor: BOP2}
$\sigma_r(\mathrm{Sq}_{d,n})$ is non-defective  if {$r\le 2n-\frac{2}{d}(n+2)$}.
\end{Corollary}
\begin{proof}
    We may assume $r\ge n+3$ and set $r=n+2+k$ with $k\ge 1$. From our hypothesis $b=k$ in Theorem \ref{Thm: BOP5.3}, therefore the inequality in \Cref{Thm: BOP5.3} is $(n+2+k)(d/2+1)\leq nd+k$, thus $k\le n-\frac{2}{d}(n+2)-2$.
\end{proof}

\section{Apolarity for squares}\label{sec:dual}
{Recall the $i$-catalecticant map for a polynomial $f\in \sym^d\CC^{n+1}$ is the linear map
$$\begin{array}{cccc}{\mathrm{Cat_{i}}}(f)\colon&\sym^{i}{(\CC^{n+1})}^\vee&\to&\sym^{d-i}\CC^{n+1}\\
    &D&\mapsto &D f\end{array}$$
    where $\sym^{i}(\CC^{n+1})^\vee = \sym^{i}({\CC^{n+1}}^\vee)$ is the space of differential operators of degree $i$. The middle catalecticant map (in case $d$ is even) is the map $\mathrm{Cat_{d/2}}(f)$.
    The subspace $f^{\perp}\colon =\oplus_{i\ge 0}\ker {\mathrm{Cat_{i}}}(f)\subset \sym({(\CC^{n+1})}^\vee) $
    is an ideal called the apolar ideal.
    An operator $g\in f^\perp$ is called apolar to $f$.

    Recall the dual variety of a projective variety $X\subset \PP^n$ is 
    $$X^\vee=\overline{\left\{H\in (\PP^n)^\vee\mid H\supset T_xX\textrm{\ for some\ }x\textrm{\ smooth point in\ }X\right\}}$$
    In this paper we work in the affine setting and consider dual varieties of affine cones, like the variety of squares $\mathrm{Sq}_{d,n}$.}

\begin{Proposition}The dual variety of $\mathrm{Sq}_{d,n}$ is the middle catalecticant hypersurface with equation $\det \mathrm{Cat}_{d/2}$.
\end{Proposition}
\begin{proof}
The tangent space to $\mathrm{Sq}_{d,n}$ at $f^2$ consists of $fg$ with any $g$ of degree $d/2$,
the hyperplanes containing such space consist of operators which are apolar to $fg$ for any $g$. Then they are apolar to $f$
(see e.g. \cite[Prop. 6.5]{OR20}),
hence their middle catalecticant vanishes.
\end{proof}

{In other words, the dual variety of $\mathrm{Sq}_{d,n}$ is the set of all degree $d/2$ forms $g\in f^\perp$, for some $f\in \sym^{d/2}\C^{n+1}$, and by bi-duality we can identify it with the forms in $\sym^{d/2}\C^{n+1}$ whose middle catalecticant is rank deficient.} 

\begin{Proposition}\label{prop:containment}The dual variety of the $k$-secant variety $\sigma_k(\mathrm{Sq}_{d,n})$, assumed not to fill the ambient space, is contained in the locus where the middle catalecticant drops rank at least by $k$.
\end{Proposition}
\begin{proof}
Assume $f=\sum_{i=1}^kf_i^2$ with $f_i$ independent forms.
Assume now a hyperplane $H$ contains the tangent space at $f$ of the $k$-secant
$\sigma_k(\mathrm{Sq}_{d,n})$.
Then by Terracini's Lemma this hyperplane contains the tangent spaces at $f_i^2$ which are 
$\{f_ig\mid g\in\sym^{d/2}\CC^{n+1}\}$,
thus $f_i$ are apolar to $H$. It follows that the middle catalecticant of $H$
has $\langle f_1,\ldots, f_k\rangle$ in the kernel. 
\end{proof}

For plane sextics the containment in \Cref{prop:containment} becomes an equality for $k=1, 2, 3$, see \cite[Section 3]{Foot}. Indeed the catalecticant ${\mathcal C}_3$
is $10\times 10$, the locus where it drops rank by $2$ has codimension $3$ (while the
$8$-secant to $\nu_6(\PP^2)$ has codimension $4$) and dual given by $\sigma_2(\mathrm{Sq}_{3,2})$.
The locus where ${\mathcal C}_3$ drops rank by $3$ has codimension $6$ (while the
$7$-secant to $\nu_6(\PP^2)$ has codimension $7$) and dual given by $\sigma_3(\mathrm{Sq}_{3,2})$.

\section{Computational verification}\label{sec:computational}

Utilizing the computer algebra system Macaulay2 \cite{M2} we were able to verify further cases of generic $\mathrm O(r)$-identifiability by applying the sufficient criterion in {\Cref{thm:sufficientcondition}}. 
\begin{Proposition}\label{prop: identM2}
    Let $r< r_g$, then $\sigma_r(\mathrm{Sq}_{d,n})$ is generically $\mathrm{O}(r)$-identifiable in the following cases
    \begin{enumerate}
        \item $n=2$.
        \item $r\leq n+1$.
        \item $n=3$, $d\leq 34$.
        \item $n=4$, $d\leq 16$.
        \item $n=5$, $d\leq 12$.
        \item $n=6$, $d\leq 14$.
        \item $n=7$, $d\leq 8$.
        \item $d=4$, $n\leq 16$.
        \item $d=6$, $n\leq 9$.
        
    \end{enumerate}
\end{Proposition}

We further verified cases of non-defectiveness of $\sigma_r(\mathrm{Sq}_{d,n})$ utilising Macaulay2. This together with \Cref{Cor: BOP1} guarantee the non-defectiveness for small $r$ when $n=4,5,6$. For the convenience of the reader we repeat in item (14) of Proposition \ref{prop: dimM2} the inequality of Corollary \ref{Cor: BOP2}.

\begin{Proposition}\label{prop: dimM2}
    Let $r\leq r_g$, then  $\sigma_r(\mathrm{Sq}_{d,n})$  has dimension equals to the expected dimension (according to \Cref{def:expdim}) in the following cases:

    \begin{enumerate}
    \item $n=2$.
        \item $r\leq n+2$.
        \item $n=3$, $d\leq 44$.
        \item $n=4$, $d\leq 30$.
        \item $n=4$, $r\leq7$.
        \item $n=5$, $d\leq 20$.
        \item n=5, $r\leq 9$
        \item $n=6$, $d\leq 14$.
        \item $n=6$, $r\leq 11$.
        \item $d=4$, $n\leq 29$.
        \item $d=6$, $n\leq 13$.
        \item $d=8$, $n\leq 10$.
        \item $d=10$, $n\leq 8$.
        \item $r\leq  
        2n-\frac{2}{d}(n+2)$.
    \end{enumerate}
\end{Proposition}

We would like to stress that during the experiments no defective or not generically $\mathrm{O}(r)$-identifiable cases have been found.

The codes utilized in M2 to verify the generic identifiability and the dimension respectively are presented next.
{In the code we produced random points $f_i$ and computed respectively the stacked Hessian matrix and the dimension of the span of their tangent spaces. Since both the rank of the stacked Hessian and the dimension of the span of their tangent spaces were maximal, by semicontinuity the same holds for generic points $f_i$,
this allows to apply our criterion
{\Cref{thm:sufficientcondition}}.}
\small
\begin{Verbatim}[commandchars=\\\{\}]
n=3;
K=\textcolor{ForestGreen}{ZZ}/101[x_0..x_n]
d=4;
m=\textcolor{RoyalBlue}{binomial}(n+sub(d/2,\textcolor{ForestGreen}{ZZ}),n);
N=\textcolor{RoyalBlue}{binomial}(n+d,n);
g=\textcolor{RoyalBlue}{floor}(N/m)-1;
b=\textcolor{RoyalBlue}{basis}(\textcolor{RoyalBlue}{sub}(d/2,\textcolor{ForestGreen}{ZZ}),K);
B=\textcolor{RoyalBlue}{basis}(d,K);
\textcolor{Fuchsia}{while} (g+1)*m-\textcolor{RoyalBlue}{binomial}(g+1,2)<\textcolor{RoyalBlue}{binomial}(n+d,n) \textcolor{Fuchsia}{do} (g=g+1);
cod=N+\textcolor{RoyalBlue}{binomial}(g,2)-g*m;
\textcolor{Fuchsia}{for} i \textcolor{Fuchsia}{from} 0 \textcolor{Fuchsia}{to} g-1 \textcolor{Fuchsia}{do} f_i=\textcolor{RoyalBlue}{random}(\textcolor{RoyalBlue}{sub}(d/2,\textcolor{ForestGreen}{ZZ}),K);
M_0=\textcolor{RoyalBlue}{matrix}\{\{f_0\}\};
\textcolor{Fuchsia}{for} i \textcolor{Fuchsia}{from} 1 \textcolor{Fuchsia}{to} g-1 \textcolor{Fuchsia}{do} M_i=M_(i-1)||\textcolor{RoyalBlue}{matrix}\{\{f_i\}\};
I_0=\textcolor{RoyalBlue}{ideal}(f_0);
\textcolor{Fuchsia}{for} i \textcolor{Fuchsia}{from} 1 \textcolor{Fuchsia}{to} g-1 \textcolor{Fuchsia}{do} I_i=I_(i-1)+\textcolor{RoyalBlue}{ideal}(f_i);
\textcolor{Fuchsia}{for} i \textcolor{Fuchsia}{from} 0 \textcolor{Fuchsia}{to} g-1 \textcolor{Fuchsia}{do}
c=b*\textcolor{RoyalBlue}{gens} \textcolor{RoyalBlue}{kernel} \textcolor{RoyalBlue}{contract}(b,M_(g-1));
\textcolor{Fuchsia}{for} i \textcolor{Fuchsia}{from} 0 \textcolor{Fuchsia}{to} g-1 \textcolor{Fuchsia}{do} sb_i=f_i*b;
A=sb_0;
\textcolor{Fuchsia}{for} i \textcolor{Fuchsia}{from} 1 \textcolor{Fuchsia}{to} g-1 \textcolor{Fuchsia}{do} A=A|sb_i;
H=B*\textcolor{RoyalBlue}{gens} \textcolor{RoyalBlue}{kernel} \textcolor{RoyalBlue}{transpose} \textcolor{RoyalBlue}{diff}(\textcolor{RoyalBlue}{transpose} B,A);
p=H_0;
\textcolor{Fuchsia}{for} i \textcolor{Fuchsia}{from} 1 \textcolor{Fuchsia}{to} cod-1 \textcolor{Fuchsia}{do} p=p+\textcolor{RoyalBlue}{random}(\textcolor{ForestGreen}{ZZ}/101)*H_i;
Hess= \textcolor{RoyalBlue}{sub}(\textcolor{RoyalBlue}{matrix} \textcolor{RoyalBlue}{apply}(m-g,i->\textcolor{RoyalBlue}{apply}(m-g,j->\textcolor{RoyalBlue}{contract}(p_0,c_(0,i)*c_(0,j)))),\textcolor{ForestGreen}{ZZ}/101);
\textcolor{RoyalBlue}{print}(\textcolor{RoyalBlue}{rank} Hess==m-g,d,g)
\end{Verbatim}

\vspace{0.2cm}

{\hrule width \textwidth}

\vspace{0.13cm}

\begin{Verbatim}[commandchars=\\\{\}]
n=3;
d=8;
m=\textcolor{RoyalBlue}{binomial}(n+\textcolor{RoyalBlue}{sub}(d/2,\textcolor{ForestGreen}{ZZ}),n);
D=\textcolor{RoyalBlue}{binomial}(n+d,n);
K=\textcolor{ForestGreen}{ZZ}/101[x_0..x_n];
g=\textcolor{RoyalBlue}{floor}(D/m)-1;
\textcolor{Fuchsia}{while} (g+1)*m-\textcolor{RoyalBlue}{binomial}(g+1,2)<\textcolor{RoyalBlue}{binomial}(n+d,n) \textcolor{Fuchsia}{do} (g=g+1);
\textcolor{Fuchsia}{for} i \textcolor{Fuchsia}{from} 0 \textcolor{Fuchsia}{to} g-1 \textcolor{Fuchsia}{do} p_i=\textcolor{RoyalBlue}{random}(\textcolor{RoyalBlue}{sub}(d/2,\textcolor{ForestGreen}{ZZ}),K)
b=\textcolor{RoyalBlue}{basis}(sub(d/2,\textcolor{ForestGreen}{ZZ}),K);
bb=\textcolor{RoyalBlue}{basis}(\textcolor{RoyalBlue}{sub}(d,\textcolor{ForestGreen}{ZZ}),K);
\textcolor{Fuchsia}{for} i \textcolor{Fuchsia}{from} 0 \textcolor{Fuchsia}{to} g-1 \textcolor{Fuchsia}{do}
\textcolor{Fuchsia}{for} j \textcolor{Fuchsia}{from} 0 \textcolor{Fuchsia}{to} m-1 \textcolor{Fuchsia}{do}
M_\{i,j\}=\textcolor{RoyalBlue}{transpose}(\textcolor{RoyalBlue}{coefficients}((p_i*b_j)_0,\textcolor{RoyalBlue}{Monomials}=>bb))_1
N=M_\{0,0\};
\textcolor{Fuchsia}{for} i \textcolor{Fuchsia}{from} 0 \textcolor{Fuchsia}{to} g-1 \textcolor{Fuchsia}{do}
\textcolor{Fuchsia}{for} j \textcolor{Fuchsia}{from} 0 \textcolor{Fuchsia}{to} m-1 \textcolor{Fuchsia}{do}
N=N||M_\{i,j\}
\textcolor{RoyalBlue}{rank} N==g*m-\textcolor{RoyalBlue}{binomial}(g,2)
\end{Verbatim}
\normalsize

\printbibliography
\end{document}